\documentclass{birkjour}
\usepackage{amsmath}
\usepackage{amsthm}
\usepackage{amsfonts}
\usepackage{url,subfigure,epsfig,enumerate}
\usepackage{graphicx}
\usepackage{breqn}

\newtheorem{theorem}{Theorem}
\newtheorem{lemma}[theorem]{Lemma}

\newtheorem{proposition}[theorem]{Proposition}
\newtheorem{definition}[theorem]{Definition}

\newtheorem{remark}[theorem]{Remark}

\begin{document}
\tolerance10000

\title{Offsets of a regular trifolium}

\author{Thierry Dana-Picard}
\address{Jerusalem College of Technology \\ Havaad Haleumi Street 21, 9116011 Jerusalem, Israel}
\email{ndp@jct.ac.il}

\author{Zoltán Kovács}
\address{Private Pädagogische Hochschule der Diözese Linz\\ Salesianumweg 3, 4020 Linz, Austria}
\email{kovzol@gmail.com}

\begin{abstract}
The non-uniqueness of a rational parametrization of a rational plane curve may influence the
process of computing envelopes of 1-parameter families of plane curves. We study envelopes of family
of circles centred on a regular trifolium  and its offsets, paying attention to different parametrizations.
We use implicitization both to show that two rational parametrizations of a curve are equivalent, 
and to determine an implicit equation for the envelope under study. 
The derivation of an implicit equation of an offset follows another path, 
leading to new developments of the package GeoGebra Discovery. As an immediate
symbolic result, we obtain that in the general case the offset curve of a regular
trifolium is an algebraic curve of degree 14. We illustrate this fact by providing
a GeoGebra applet that computes such curves automatically and visualizes them in a web browser.

\end{abstract}



\maketitle

\section{Envelopes and offsets}

Let $\mathcal{C}_t$ be a family of plane curves, parameterized by the real $t$, given by the equation $F(x,y,t)=0$.

\begin{definition}[\emph{Impredicative}; \cite{kock}]
A plane curve $\mathcal{E}$ is called an \emph{envelope} of the family $\{ \mathcal{C}_k \}$  if the following properties hold:
\begin{enumerate}[(i)]
\item every curve  in the family $\{ \mathcal{C}_t \}$ is tangent to $\mathcal{E}$;
\item	to every point $M$ on $\mathcal{E}$ is associated a value $t(M)$ of the parameter $t$, such that $\mathcal{C}_t$ is tangent to $\mathcal{E}$ at the point $M$;
\item The function $t(M)$ is non-constant on every arc of $\mathcal{E}$.
\end{enumerate}
\end{definition}

Two other definitions exist, also mentioned by  \cite{kock}. One of them reads as follows:

\begin{definition}[\emph{Synthetic}; \cite{kock}]
The envelope of the given family of curves is the set of limit points of intersections of nearby curves $C_{t}$.
\end{definition}
 This definition is somehow problematic, as it would request the definition of a topological set of curves, in order to have a reasonable definition of a limit. Nevertheless, such a point of view has been used for a simple case by \cite{DPZ-envelopes}.

The most used and usable definition is called \emph{analytic} by \cite{kock}; note that \cite{berger}, (sections 9.6.7 and 14.6.1) gives this one only. It has been derived form the limit definition as  theorem by \cite{DPZ-envelopes}.
\begin{definition}[\emph{Analytic}; \cite{kock}]
\label{def analytic env}
Let $C_{t}$ be a 1-parameter family of curves given by an equation $F(x,y,t)=0$. The envelope of the family $\{ C_{t} \}$  is the set of all points $(x, y)$ such that there is a $t$ verifying the system of equations
\begin{equation}
\begin{cases}
F(x, y, t) = 0,\\
\frac{\partial F(x, y, t)}{\partial t} = 0.
\end{cases}
\label{def envelope system of equations}
\end{equation}
\end{definition}

Following a Wikipedia page\footnote{\url{http://en.wikipedia.org/wiki/Envelope_(mathematics)}}, \cite{br} reports on a $4^{\textrm{th}}$ definition of an envelope, namely: the envelope is the curve that bounds the planar region described by the points belonging to the curves in the family. In the case of our study, this is the \emph{offset} of the trifolium at distance 1, as coined by \cite{sendra}, page 10. Actually, the first three definitions mentioned above are equivalent, but the last one is not equivalent to them. There exist examples where envelope and offset are identical, 
such as in  \cite{DPZ-envelopes} and \cite{sendra}, but other examples enlighten the difference; see \cite{dp-circles on astroid}.
The family studied in this paper is an example with new properties. They are studied in section \ref{section offset}.

Note that the difference between envelope and offset is of the utmost importance for industrial applications, such as the determination of safety zones in industrial plants and entertainment parks.
This issue is evoked in \cite{dp-circles on astroid}.

\section{A regular trifolium: different presentations}
\label{section trifolium}
We study various constructions on the basis of a regular trifolium  (called also regular \emph{trefoil}). We recall the basic definitions, referring to \cite{mathcurve trifolium}, where other constructions are also proposed.

\begin{definition}
A \emph{regular trifolium} is a plane algebraic curve, given by the implicit equation
\begin{equation}
\label{implicit equation trifolium}
\left(x^2+y^2\right)^2=ax\left(x^2-3y^2\right)
\end{equation}
where $a$ is a positive real parameter
\end{definition}
The parameter $a$ influences the \emph{size} of the trifolium, not its shape. Two examples are shown in Figure \ref{two examples of trifolium}, obtained using GeoGebra, a Dynamic Geometry System (freely downloadable from \url{http://geogebra.org}). An important reason for using this software is that it enables dynamical mouse-driven experiments. Animation can be obtained with a CAS also, requiring some programming.
The differences between the two kinds of animation
are discussed in \cite{dp-circles on astroid}.
\begin{figure}[htbp]
\centering
\mbox{
       \subfigure[$a=1$]{\epsfig{file=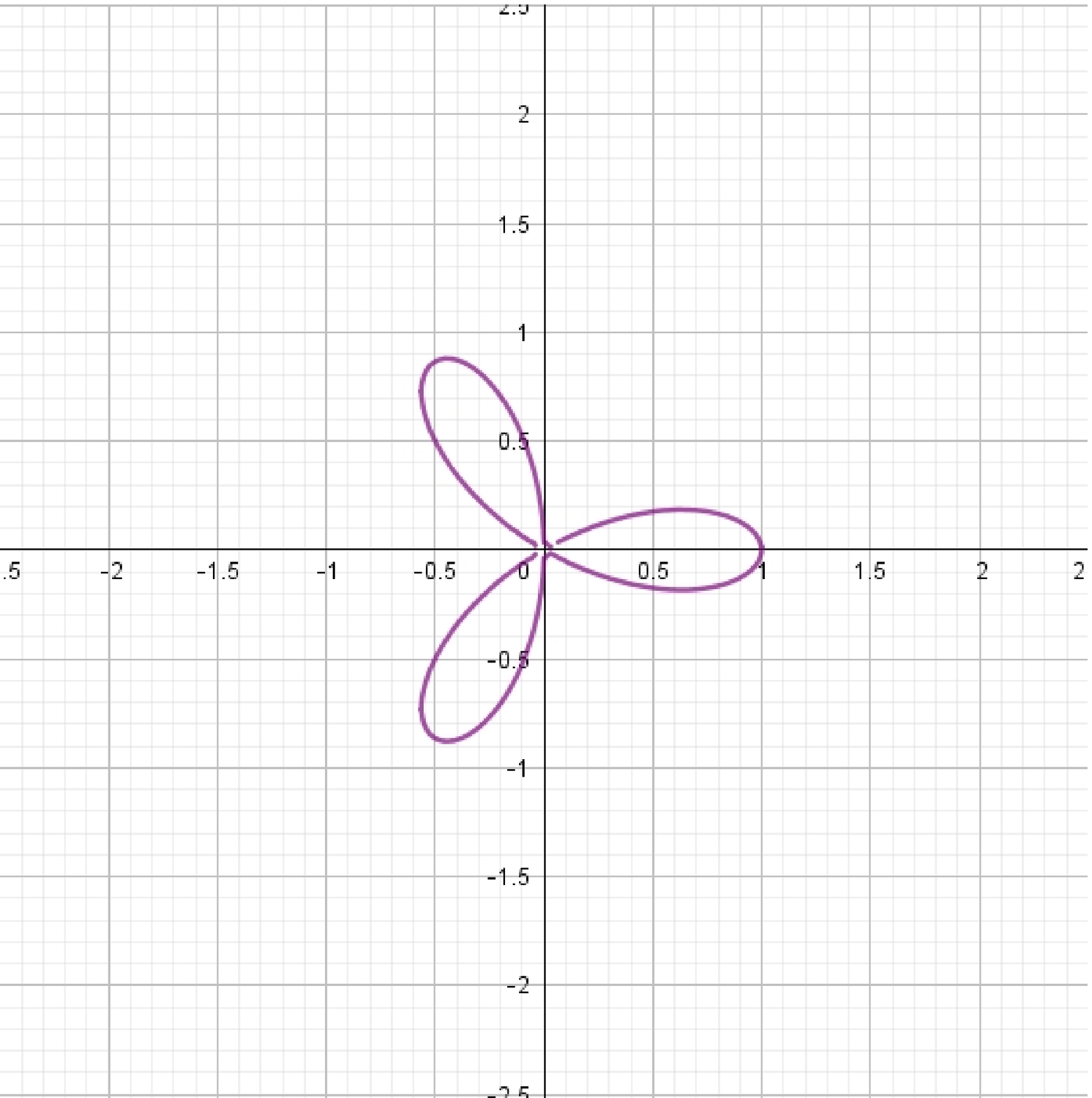,width=4.5cm}}
       \qquad
       \subfigure[$a=2$]{\epsfig{file=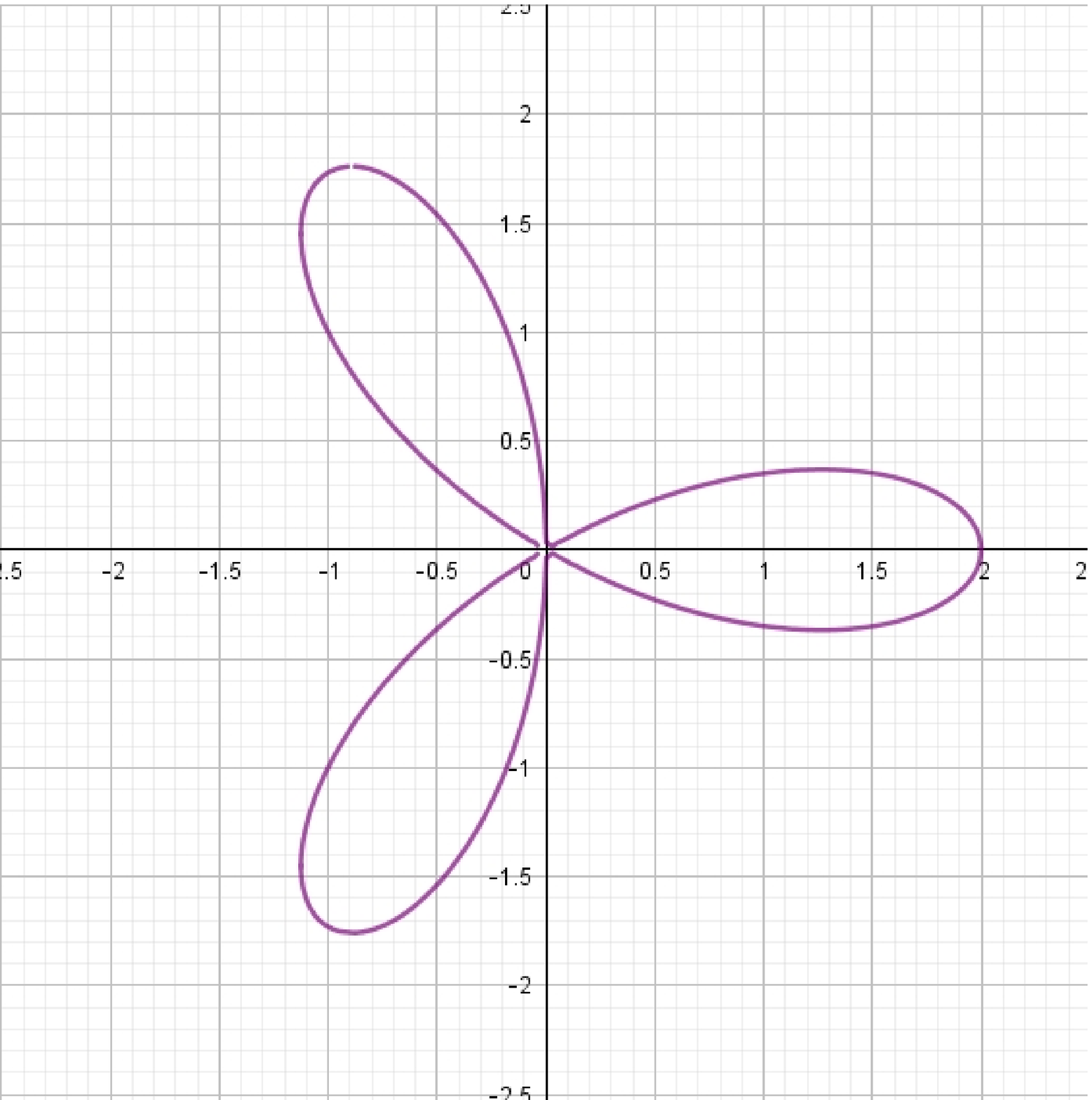,width=4.5cm}}
       }
\caption{Two examples of a regular trifolium}
\label{two examples of trifolium}
\end{figure}
A plot can be obtained also using an \textbf{implicitplot} command of a Computer Algebra System. WLOG, we will work in the case where $a=1$.

A more dynamical definition is as follows:
\begin{definition}
A regular trifolium is the trajectory of the second intersection point between a line and a circle turning around one of their common points in the same direction, and the circle turning four times as fast as the line.
\end{definition} 	
We will use this definition in order to derive a rational parametrization of the regular trifolium, whence showing that this curve is a rational curve.

Consider the circle $C$ through the origin and whose center has coordinates $(1/2,0)$. Its implicit equation is $(x-1/2)^2+y^2=1/4$. A general rotation matrix about the origin is 
The rotated circle $C_s$ is thus given by the equation
\begin{equation}
\label{eq rotated circle}
\left(x \cos s + y\sin s - \frac 12 \right)^2 + (y \cos s - x \sin s)^2 =\frac 14.
\end{equation}

Now consider the line $L:y=x$. Rotating the line at a speed one fourth of the rotation speed of the circle, we have a line $L_s$ whose equation is
\begin{equation}
\label{eq rotated line}
\cos \frac {s}{4} x+ y \sin\frac {s}{4} =0.
\end{equation}
The circle $C_s$ and the line $L_s$ intersect at two points, the origin and a second point whose coordinates are given as follows (for the sake of simplicity, we substitute $t=s/4$):
\begin{equation}
\label{eq intesect rotated circle and line}
\begin{cases}
x(t)=4\sin ^4 t -3\sin^2t,\\
y(t)= -\sin t \cos t \left(4\sin^2t-3\right).
\end{cases}
\end{equation}
Using this parametrization, we obtain Figure \ref{eq trifolium as rotated intersect}; an experimental checking using a DGS or a CAS shows that this plot is identical to Figure \ref{two examples of trifolium}a.
\begin{figure}[htbp]
\centering
\epsfig{file=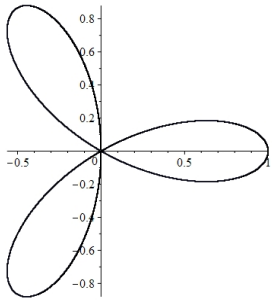,height=4cm}
\caption{Geometric construction of a regular trifolium}
\label{eq trifolium as rotated intersect}
\end{figure}
This identity may be checked by substitution of the two expressions for $x(t)$ and $y(t)$ into Equation (\ref{implicit equation trifolium}). Note that the parametrization (\ref{eq intesect rotated circle and line}) is not rational. In order to have benefit of Theorem 4.7 in \cite{sendra}, we need a rational presentation.

\cite{mathcurve rational quartic} proposes another trigonometric parametric presentation (we modify it in order to be coherent with our previous presentations):
\begin{equation}
\begin{cases}
x=a\cos t +\cos 2t,\\
y=a\sin t -\sin 2t.
\end{cases}
\end{equation}
This parametrization shows that a regular trifolium is a special case of a \emph{hypotrochoid}.
It can be verified, either experimentally with software or using trigonometric manipulations, that for $a=1/2$, we obtain the same curve studied earlier. A regular trifolium is a special case of trifolium, in which the three leaves are isometric, two of them being obtained from the third one by a rotation of angle $\frac{2\pi}{3}$ or  $\frac{4\pi}{3}$ (see \cite{dp-trifolium}).

The  following result is well-known, we emphasize it because of its importance for the subsequent sections. We will show two proofs.
\begin{proposition}
A regular trifolium is a rational curve.
\end{proposition}
This proposition means that a regular trifolium has  a rational parametrization. It is well-known that such a parametrization, when it exists, is not unique. We derive now two parametrizations of a regular trifolium. The first one will be obtained when using the following Lemma.

\begin{lemma}
\label{lemma rational param unit circle}
A rational parametrization of the unit circle centred at the origin is as follows:
\begin{equation}
\label{rational param unit circle}
\begin{cases}
x=\frac{1-t^2}{1+t^2},\\
y=\frac{2t}{1+t^2}.
  \end{cases}
\end{equation}
\end{lemma}
\begin{proof}
The proof is straightforward; we give it as a promo, because we will use the same method later in a more complicated setting.
Consider the unit circle $\mathcal{U}$ centred at the origin and a line $\mathcal{L}$  through the point $B(-1,0)$ (Figure \ref{unit circle param}).
\begin{figure}[htbp]
\centering
\epsfig{file=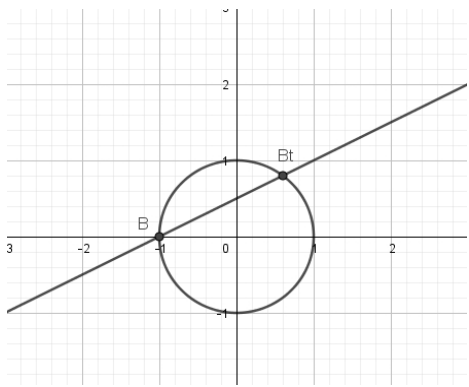,height=5cm}
\caption{Geometric construction for parameterizing the unit circle $\mathcal{U}$}
\label{unit circle param}
\end{figure}
 The tangent to $\mathcal{U}$ through $B$ is parallel to the $y$-axis, all other lines through $B$ have a slope $t$ and intersect $\mathcal{U}$ at a $2^{\textrm{nd}}$ point $B_t$. The coordinates of $B_t$ are the solutions of the system of equations
\begin{equation*}
\begin{cases}
x^2+y^2=1,\\
y=t(x+1).
\end{cases}
\end{equation*}
The result is obtained by substitution. Note that this parametrization is valid for every point but $B$.
\end{proof}

Following \cite{sendra} (p.~90), we recall that a rational parametrization can be stated by means of rational maps. Let $\mathcal{C}$ be a rational affine curve over the field
$\mathbb{K}$ and $P(t) \in \mathbb{R}(t)^2$ be a rational parametrization of $\mathcal{C}$. This parametrization induces a rational map
  \begin{center}
    \begin{tabular}{cccc}
    $P:$ & $\mathbb{A}^1(\mathbb{K})$ &$\longrightarrow$ & $\mathcal{C}$\\
      $\qquad$ & $t$ & $\mapsto$ & $P(t)$
    \end{tabular}
  \end{center}
and $P(\mathbb{A}^1(\mathbb{K}))$ is a dense subset of $\mathcal{C}$ (dense in the sense of Zariski topology).
A trigonometric parametrization of the  unit circle centred at the origin is given by $P(u)=(\cos u, \sin u), \; u \in \mathbb{R}$. On the one hand, this parametrization is a bijection between the unit circle and every interval $[2k\pi,2(k+1)\pi)$.  On the other hand, Equations (\ref{rational param unit circle}) define a bijection from
$\mathbb{R}$ onto the unit circle but the point $B$. Therefore, for almost every real $u$, there exists a real $t$ such that $\cos u = \frac{1-t^2}{1+t^2}$ and $\sin u = \frac{2t}{1+t^2}$.
By substitution into the trigonometric parametrization of the curve given above, and after simplification, we obtain the following result:
\begin{proposition}[First rational parametrization]
A rational parametrization of the regular trifolium is given by
\begin{equation}
\label{1st rational param trifolium}
\begin{cases}
x(t)=-\frac{4t^2(3t^2 - 1)(t^2 - 3)}{(t^2 + 1)^4},\\
y(t)=-\frac{2t(t - 1)(t + 1)(3t^2 - 1)(t^2 - 3)}{(t^2 + 1)^4}.
\end{cases}
\end{equation}
\end{proposition}

Another rational parametric presentation for the regular trifolium can be obtained, not applying directly Lemma \ref{lemma rational param unit circle}, but using the method of its proof.
\begin{proposition}[Second rational parametrization]
\label{prop 2nd rational param trifolium}
A rational parametrization of the regular trifolium is given by
\begin{equation}
\label{eq 2nd rational param trifolium}
\begin{cases}
x(t) = -\frac{3t^2 - 1}{(t^2+ 1)^2},\\
y(t) = -\frac{t(3t^2 - 1)}{(t^2+ 1)^2}.
\end{cases}
\end{equation}
\end{proposition}
\begin{proof}
The regular trifolium is a quartic with a triple point at the origin. The $y-$axis is tangent to the curve at the origin; by substitution of $x=0$ into Equation \ref{implicit equation trifolium}, we prove that there is no other point of intersection.   Any other line through the origin (with equation $y=tx$, the parameter $t$ being the slope of the line) through the origin intersects the trifolium at one extra point, as we see from the solution of system of equations \ref{system eq for param trifolium} ; see Figure \ref{regular trifolium - 2nd param}.
\begin{figure}[htbp]
\centering
\epsfig{file=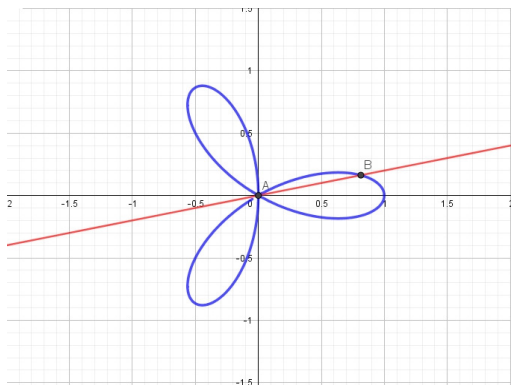,height=5cm}
\caption{Geometric construction for parameterizing the regular trifolium}
\label{regular trifolium - 2nd param}
\end{figure}
The coordinates of this point are the solutions of the following system of equations:
\begin{equation}
\label{system eq for param trifolium}
\begin{cases}
(x^2+y^2)^2=x \; (x^2-3y^2),\\
y=tx.
\end{cases}
\end{equation}
The solution is obtained by substitution and simplification\footnote{Of course, these computations may be performed using a CAS with the \textbf{solve} command. The output may show both the origin and the extra (variable) point.}.
\end{proof}

\begin{remark}
Other rational parametrizations of the unit circle can be obtained, using other families of lines, such as the families given by equations $y=t(x-1)$ (lines through $A(1,0)$, or $y-1=tx$ (lines through $B(0,1)$, etc.). Following the same process as above, we obtain different rational presentations for the regular trifolium. Nevertheless, it can be proven that all of them provide the same implicit equation for the envelope that we will study.
\end{remark}

\section{The 1-parameter family of unit circles centred on the regular trifolium}

We consider now the family of unit circles $\mathcal{C}_t$ centred on the regular trifolium described in Section \ref{section trifolium}. The index $t$ is actually the parameter that we use for the parametric presentation of the trifolium. Here we will use the formulas derived in Prop.~\ref{prop 2nd rational param trifolium}, namely:
\begin{equation*}
\begin{cases}
x(t) = -\frac{3t^2 - 1}{(t^2+ 1)^2},\\
y(t) = -\frac{t(3t^2 - 1)}{(t^2+ 1)^2}.
\end{cases}
\end{equation*}
A general equation for the family of unit circles  centred on the trifolium is of the form $F(x,y,t)=0$, where
\begin{equation*}
  F(x,y,t)=\left( x+ \frac{3t^2 - 1}{(t^2+ 1)^2} \right)^2+\left( y+\frac{t(3t^2 - 1)}{(t^2+ 1)^2}\right)^2-1.
\end{equation*}
According to Definition \ref{def analytic env}, an envelope of this family of unit circles, if it exists, is determined by the solutions of the following system of equations (the second one is displayed after simplification):
\begin{equation}
\label{eq analytic def env}
\begin{cases}
\left( x+ \frac{3t^2 - 1}{(t^2+ 1)^2} \right)^2+\left( y+\frac{t(3t^2 - 1)}{(t^2+ 1)^2}\right)^2-1 =0,\\
-6t^6y + (-12x - 18)t^5 + 18t^4y + (8x + 60)t^3 + 22t^2y + (20x - 18)t - 2y = 0.
\end{cases}
\end{equation}
Denote
\begin{equation*}
D(t)=\left(9t^6 - 45t^4 + 75t^2 + 1\right)\left(t^2 + 1\right)^2.
\end{equation*}
Then solutions of the System (\ref{eq analytic def env}) are given by the two following parametrizations, each one defining an envelope of the family of circles, and their union being also an envelope of this family:
\begin{equation}
\label{param eq envelope 1}
\begin{cases}
x(t)= \frac{27\; \left[ (-3t^4 + 12t^2 - 1)\; \sqrt{D(t)(3t^2 - 5)^2(t^2 + 1)} - 81t^{10} + 567t^8 - 1530t^6 + 1566t^4 - 357t^2 - 5 \right]}{D(t) \; (3t^2-5)}, \\
y(t)=\frac{t \left[ -27t^8 + 144t^6 - 270t^4 + 72t^2 + 2\sqrt{D(t)(3t^2 - 5)^2(t^2 + 1)} + 1 \right]}{D(t)},
\end{cases}
\end{equation}
and
\begin{equation}
\label{param eq envelope 2}
\begin{cases}
x(t)=\frac{27 \; \left[ (3t^4 - 12t^2 + 1)\sqrt{D(t)(3t^2 - 5)^2(t^2 + 1)} - 81t^{10} + 567t^8 - 1530t^6 + 1566t^4 - 357t^2 - 5 \right] }{ D(t) \; (3t^2-5)},\\
y(t)=-\frac{t \; \left[ (27t^8 - 144t^6 + 270t^4 - 72t^2 + 2\sqrt{D(t)(3t^2 - 5)^2(t^2 + 1)} - 1) \right]}{D(t)}.
\end{cases}
\end{equation}
Figures \ref{fig 1st component of envelope} and \ref{fig 2nd component of envelope} show the two separate components, which will be denoted by $\mathcal{E}_1$ and $\mathcal{E}_2$ respectively.  In Figure \ref{fig trifolium and envelope}, both the trifolium and the ``full'' envelope $\mathcal{E}=\mathcal{E}_1 \cup \mathcal{E}_2$ are displayed. A first analysis of the display reveals the following properties:
\begin{itemize}
\item The roles of the components $\mathcal{E}_1$ and $\mathcal{E}_2$  are not symmetric.
\item It seems that $\mathcal{E}_1$ is twice tangent to $\mathcal{C}$ and  $\mathcal{E}_2$ is tangent once.
\item The tree points of tangency are the vertices of an equilateral triangle.
\end{itemize}
Of course, these remarks have to  be proven. This can be done by algebraic manipulations and CAS aided solution of systems of equations.
\begin{figure}[h]
\centering
\mbox{
       \subfigure[1st component]{\epsfig{file=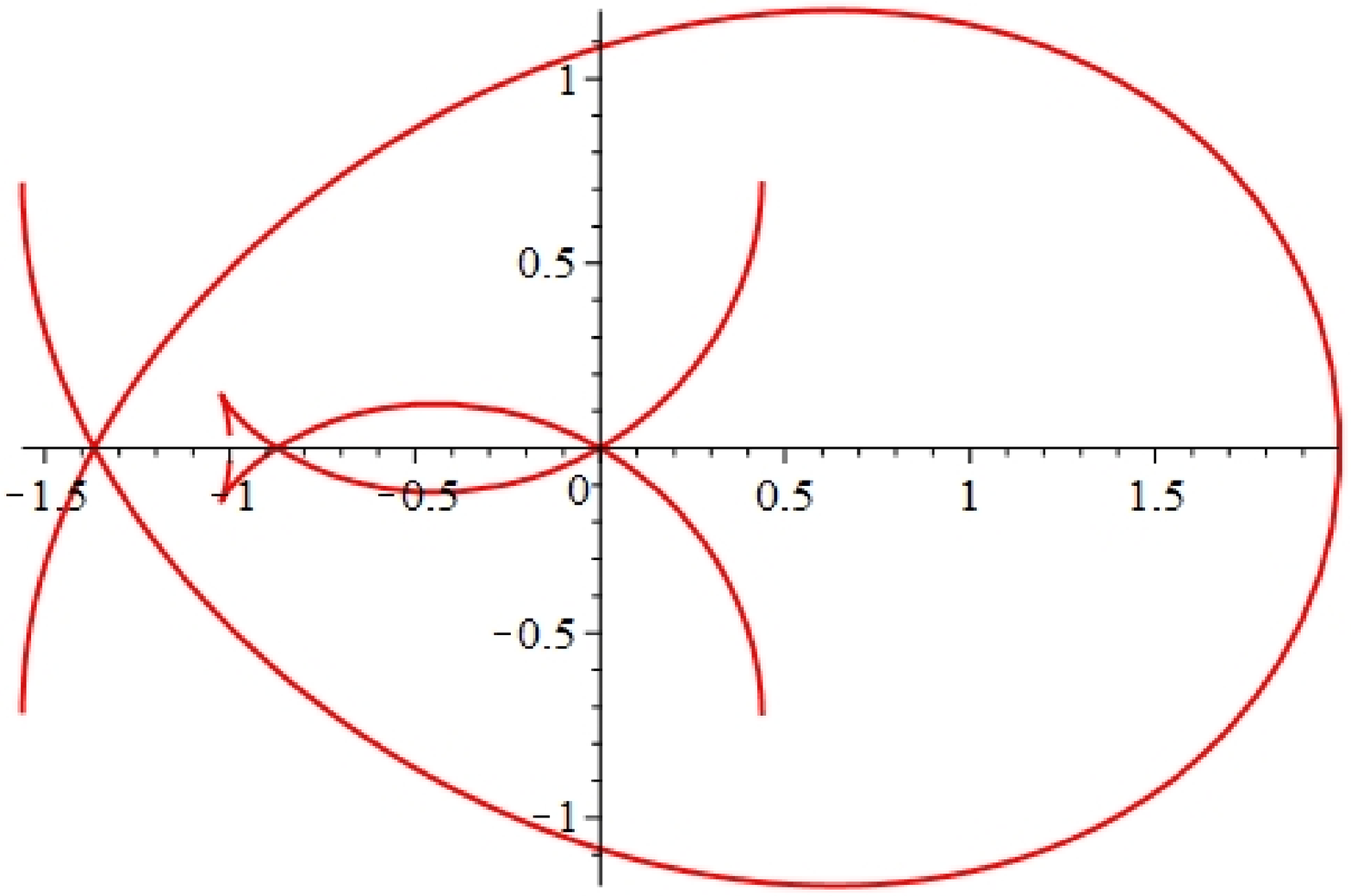,width=3.5cm}
       \label{fig 1st component of envelope}}
       \ 
       \subfigure[2nd component]{\epsfig{file=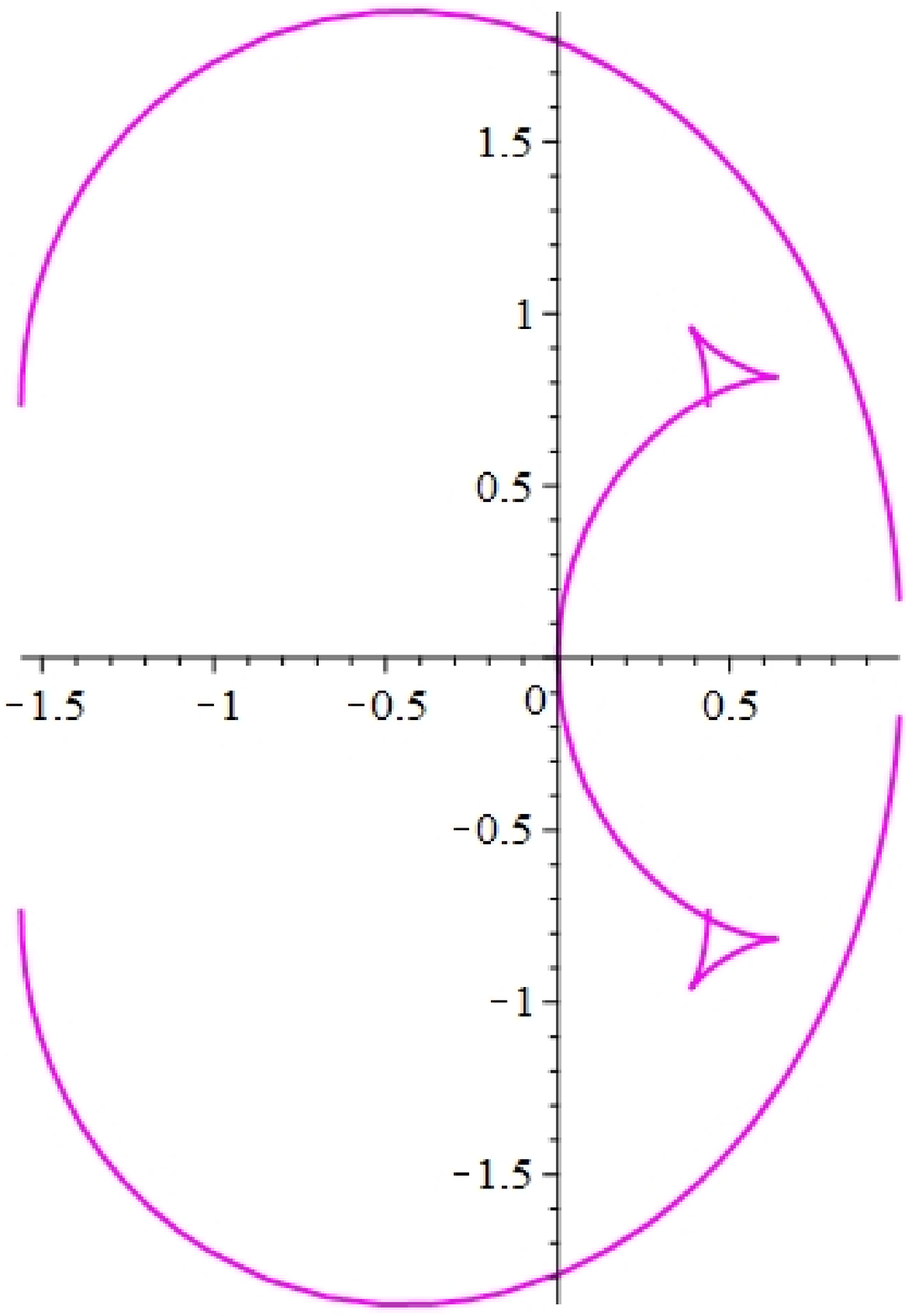,width=3.5cm}
       \label{fig 2nd component of envelope}}
       \ 
       \subfigure[Full envelope]{\epsfig{file=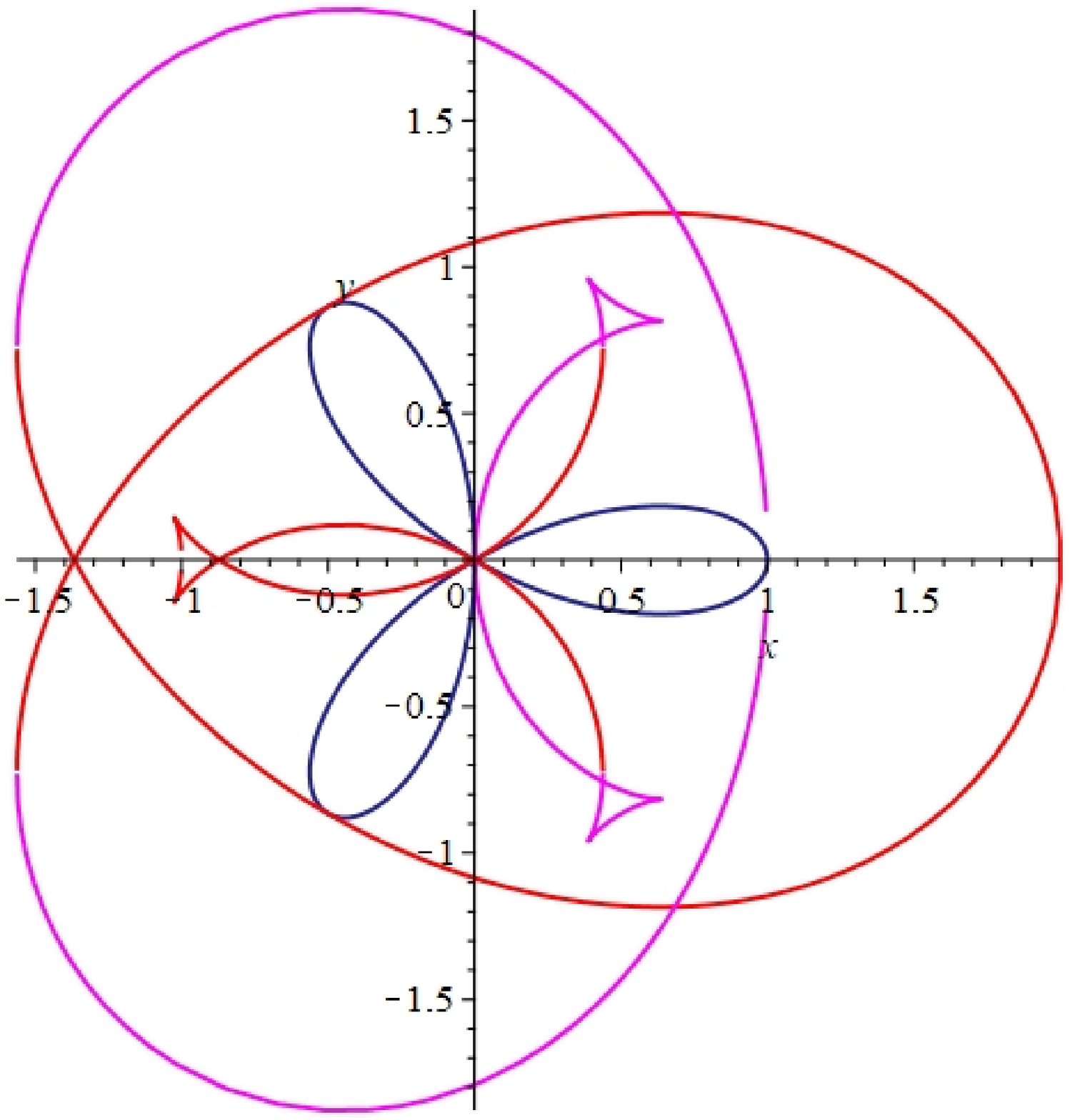,width=4cm}
       \label{fig trifolium and envelope}}
     }
\caption{The envelopes of the family of unit circles}
\label{fig 3 envelopes}
\end{figure}

\begin{remark}
The first and second envelopes are themselves disjoint unions of subcurves, with obvious discontinuities. These plots have been obtained with Maple. A first attempt with the \textbf{plot} command yielded within one second a picture with superfluous straight segments connecting the separate subcurves. Addition of the option \textbf{discont=true} eliminated these superfluous, but each computation required about one minute.
\end{remark}

The full envelope is also the disjoint union of two components: an external one (with three bulbs) and an internal one (looking like 3 fishes connected by their mouth). We will now use algebraic manipulations in order to determine an implicit polynomial equation for the envelope. The properties of the polynomial will indicate whether each of these components is an algebraic curve, or of they are impossible to distinguish by that way.

We transform Equations (\ref{param eq envelope 1}) into polynomial equations, using squaring and transferring form side to side. We obtain the two following polynomials in three variables $x,y,t$:
\begin{dmath*}
P_1 (x,y,t)  = 9t^{14}x^2 - 9t^{14} - 9t^{12}x^2 + 54t^{12}x + 45t^{12} - 51t^{10}x^2 - 180t^{10}x
 + 120t^{10} + 67t^8x^2 + 18t^8x- 678t^8 + 283t^6x^2 + 648t^6x + 647t^6
 + 261t^4x^2 + 250t^4x - 567t^4 + 79t^2x^2  - 148t^2x + 90t^2 + x^2 - 2x,
\end{dmath*}
\begin{dmath*}
P_2 (x,y,t) = 9t^{14}y^2 + 54t^{13}y - 9t^{12}y^2 + 45t^{12} - 180t^{11}y - 51t^{10}y^2 - 447t^{10}
   + 18t^9y + 67t^8y^2 + 1106t^8 + 648t^7y + 283t^6y^2 - 462t^6 + 250t^5y
  + 261t^4y^2 - 111t^4 - 148t^3y + 79t^2y^2 - 99t^2 - 2ty + y^2.
\end{dmath*}

These polynomials generate an ideal $J=\left<P_1,P_2\right>$ in the polynomial ring $\mathbb{R}[x,y,t]$. Using Maple's command \textbf{EliminationIdeal}, we eliminate the variable $t$ (i.e.~we compute a projection onto
$\mathbb{R}[x,y]$). The obtained ideal is generated by a polynomial of degree 28 in two variables $x$ and $y$. This polynomial has exactly two irreducible components, each of degree 14, namely:

\begin{dmath*}
F_1(x,y) = 256x^{14} + 1792x^{12}y^2 + 5376x^{10}y^4 + 8960x^8y^6 + 8960x^6y^8 + 5376x^4y^{10}
 + 1792x^2y^{12} + 256y^{14} - 512x^{13} - 1024x^{11}y^2 + 2560x^9y^4 + 10240x^7y^6
 + 12800x^5y^8 + 7168x^3y^{10} + 1536xy^{12} + 3840x^{12} - 12032x^{10}y^2 - 78848x^8y^4
 - 122368x^6y^6 - 71936x^4y^8 - 8960x^2y^{10} + 3584y^{12} - 5728x^{11} + 14752x^9y^2
 + 50752x^7y^4 + 6464x^5y^6 - 51680x^3y^8 - 27872xy^{10} + 6080x^{10} + 120064x^8y^2
 + 336000x^6y^4 + 333056x^4y^6 + 107968x^2y^8 - 3072y^{10} - 5952x^9 - 22688x^7y^2
 - 435936x^5y^4 - 288480x^3y^6 + 75424xy^8 - 56823x^8 - 376540x^6y^2
 + 232758x^4y^4 - 358620x^2y^6 - 26359y^8 + 35646x^7 + 360802x^5y^2 + 636938x^3y^4
 + 201190xy^6 + 65109x^6 - 176542x^4y^2 - 256867x^2y^4 - 19328y^6 - 38090x^5 -
  387676x^3y^2- 68626xy^4 - 15173x^4+ 339668x^2y^2 + 25677y^4 + 14394x^3 - 67846xy^2
 - 3289x^2 - 90y^2 + 242x,
\end{dmath*}
\begin{dmath*}
F_2(x,y)  =1024x^{14} + 7168x^{12}y^2 + 21504x^{10}y^4 + 35840x^8y^6 + 35840x^6y^8 + 21504x^4y^{10}
 + 7168x^2y^{12} + 1024y^{14} - 2048x^{13} - 4096x^{11}y^2 + 10240x^9y^4 + 40960x^7y^6
 + 51200x^5y^8 + 28672x^3y^{10} + 6144xy^{12} - 3072x^{12} - 27648x^{10}y^2 - 67584x^8y^4
 - 71680x^6y^6 - 39936x^4y^8 - 15360x^2y^{10} - 4096y^{12} + 13632x^{11}+ 13632x^9y^2
  - 81792x^7y^4 - 190848x^5y^6 - 149952x^3y^8 - 40896xy^{10} - 4992x^{10} - 28416x^8y^2
  - 54528x^6y^4- 49152x^4y^6 - 23424x^2y^8 - 5376y^{10} - 34944x^9 + 15552x^7y^2
 + 152640x^5y^4 + 312384x^3y^6 + 99648xy^8 + 28793x^8 + 258596x^6y^2 + 220566x^4y^4
 + 35492x^2y^6 + 44729y^8 + 35078x^7 - 35078x^5y^2 - 175390x^3y^4 - 105234xy^6 - 60633x^6
 - 297828x^4y^2 - 104613x^2y^4- 73514y^6 - 35046x^5 + 70092x^3y^2 + 105138xy^4 + 38880x^4
 + 77760x^2y^2 + 38880y^4+ 23328x^3 - 69984xy^2.
\end{dmath*}

An implicit plot of $F_2$ is identical to the envelope displayed in Figure \ref{fig trifolium and envelope}. Plotting $F_1$ yields a totally different figure, which is irrelevant to our purpose. See \cite{br} for a discussion of irrelevant components in such situations.

The same process, starting from the parametric equations of the second component leads to the same two polynomials $F_1(x,y)$ and $F_2(x,y)$.

The role of the two components can be explored using an animation with Maple. Two snapshots are displayed in Figure \ref{fig animation}.
\begin{figure}[h]
\centering
\mbox{
       \subfigure[]{\epsfig{file=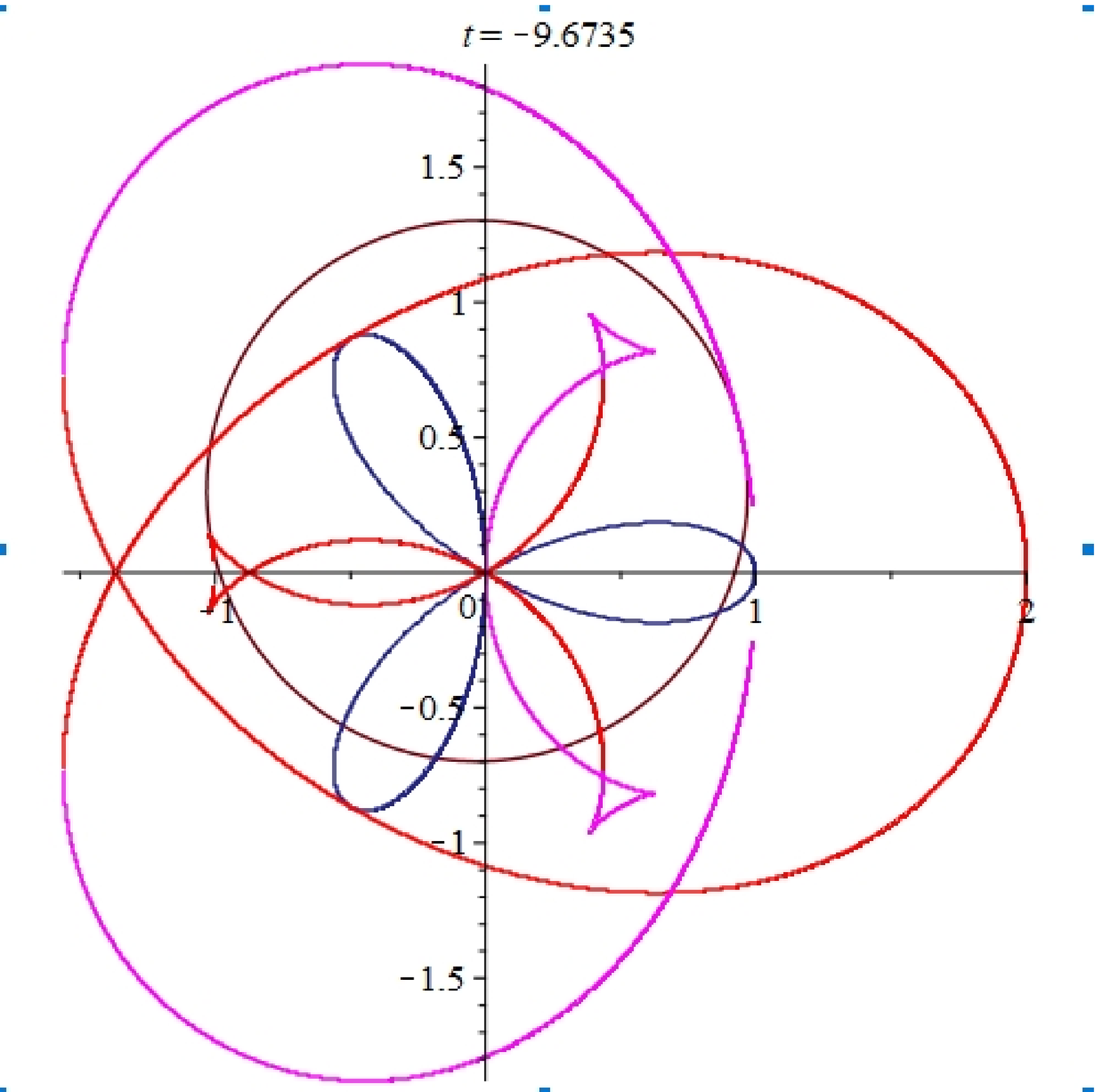,width=4.5cm}}
              \qquad \quad
       \subfigure[]{\epsfig{file=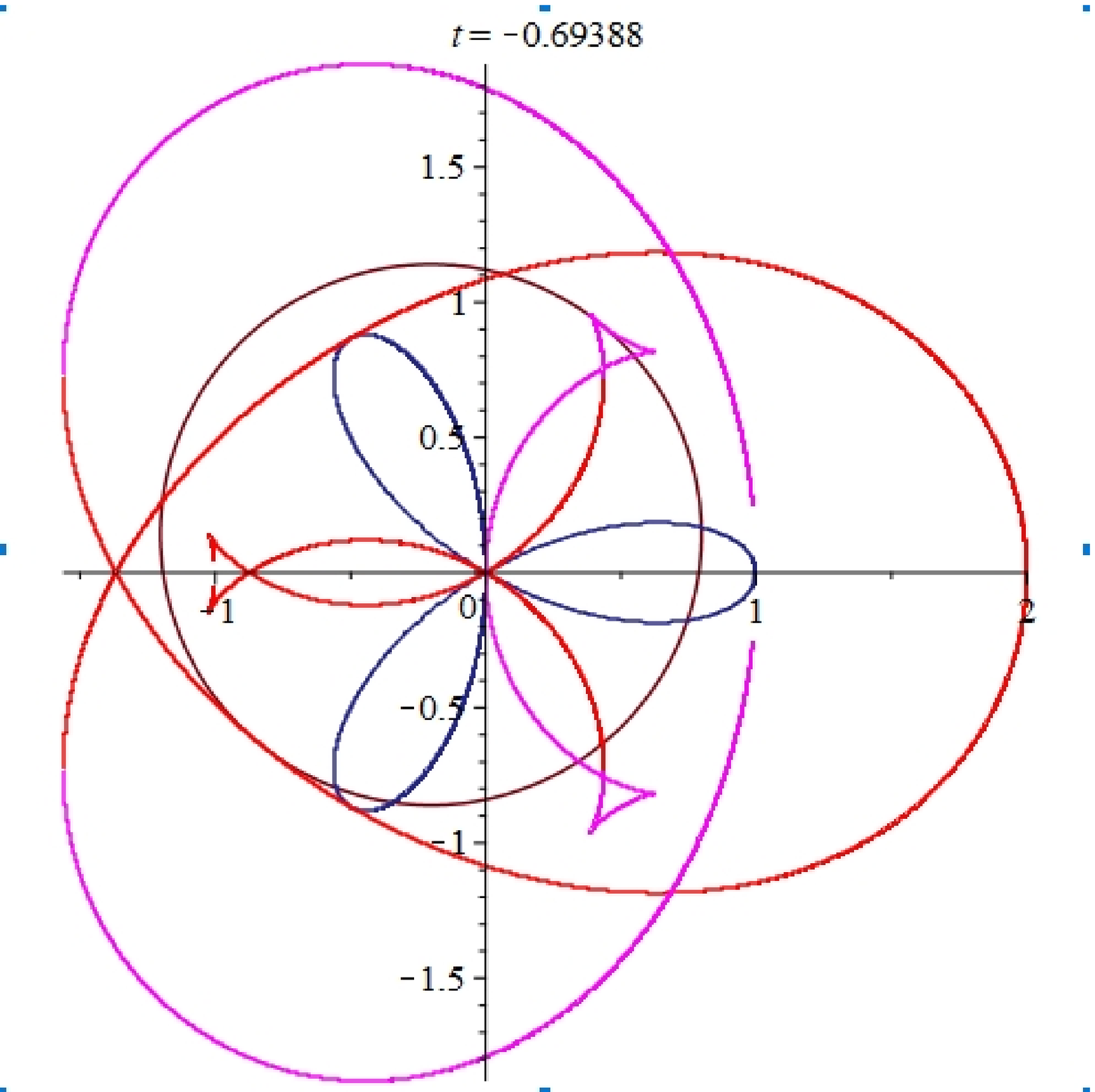,width=4.5cm}}
     }
\caption{Animation of the family of circles when the center moves along the trifolium}
\label{fig animation}
\end{figure}
Note that for every value of the parameter $t$, the corresponding circle touches both components. This confirms, as if this was necessary, an interesting detail in the definition of an envelope: its non-uniqueness. Here, each component is a separate envelope of the given family of circles, and their union too.

\section{The offset at distance 1 of the regular trifolium}
\label{section offset}
An interactive study with GeoGebra of the given family of unit circles centred on the regular trifolium  shows that the offset of the family is different from the envelope. Recall that the circles of the family have the following general equation:
\begin{equation}
\left(x+\frac{3t^2-1}{\left(t^2+1\right)^2}\right)^2+\left(y+\frac{t\left(3t^2-1\right)}{\left(t^2+1\right)^2}\right)^2=1.
\end{equation}
Using a slider bar with the \textbf{Trace On} feature to modify the values of the parameter $t$, a certain number of circles can be plotted. Note that no preview of the envelope is visible here,  but some preview of the offset is obtained. See Figure \ref{fig interactive offset preview}, which shows the regular trifolium $\mathcal{C}$ and a certain number of circles of the family.
\begin{figure}[h]
\centering
\mbox{
       \subfigure[animation of circles]{\epsfig{file=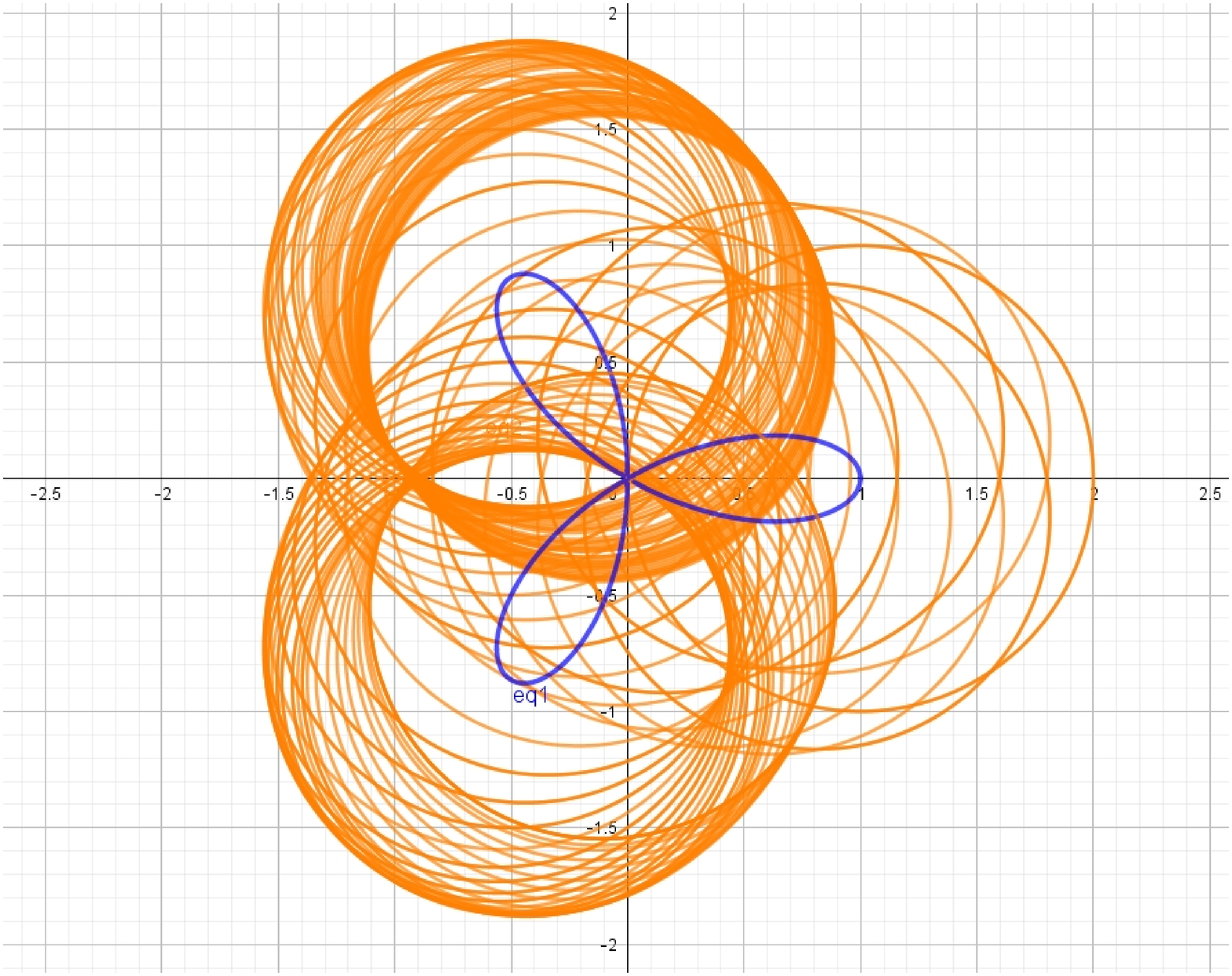,width=5cm}
       \label{fig preview offset}}
       \qquad \quad
       \subfigure[with the envelope]{\epsfig{file=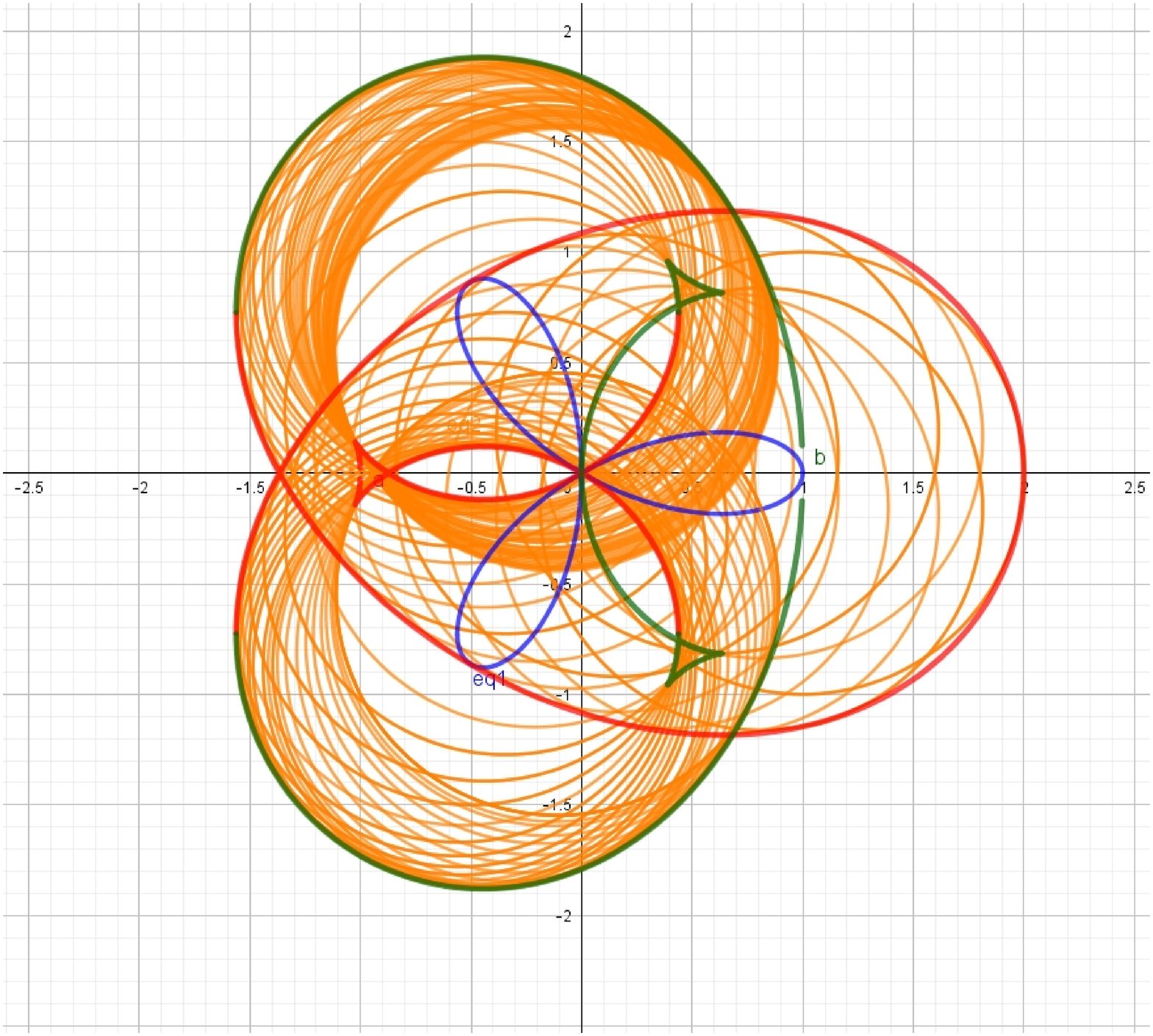,width=5cm}
       \label{fig preview offset with envelope}}
      }
\caption{Offset preview of the family of unit circles}
\label{fig interactive offset preview}
\end{figure}

From this experiment, it appears that the offset is the union of arcs of the external envelope (the internal part, looking like three fishes, is irrelevant to the question here). The open question is: which arcs? The answer can come from an analytic treatment only.

\section{Symbolic experiments with GeoGebra}

Recent improvements in GeoGebra, in particular in its experimental version
\emph{GeoGebra Discovery}, allow us to directly study the offsets
of the regular trifolium  at various distances. Latest versions since December 2020 (freely available
at \url{https://github.com/kovzol/geogebra/releases})
make it possible to define the algebraic equation
(\ref{implicit equation trifolium}) with parameter $a=1$, and then use it as a path for a moving circle.
First the user enters
this equation in the \emph{Input Bar} and obtains the object $eq1$,
then a distance $f$ will be defined by a segment $AB$.
As a next step, a point $C$ is attached to the trifolium curve, and a circle $c$ with center $C$
and radius $f$ will be added. At last, the command \texttt{Envelope($c$,$C$)}
will be used to compute the offset of the trifolium at distance $f$. (The same result
can be obtained by using the \emph{Envelope tool}
\raisebox{-.15\height}{\includegraphics[width=0.4cm]{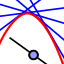}}
 and clicking on $c$ and $C$.
In this alternative way, the user does not even have to type anything but the equation, and just use the mouse.)
Fig.~\ref{trifolium-exp1} shows how the object $eq2$ can be obtained finally when $f=0.5$.

\begin{figure}[htbp]
\centering
\epsfig{file=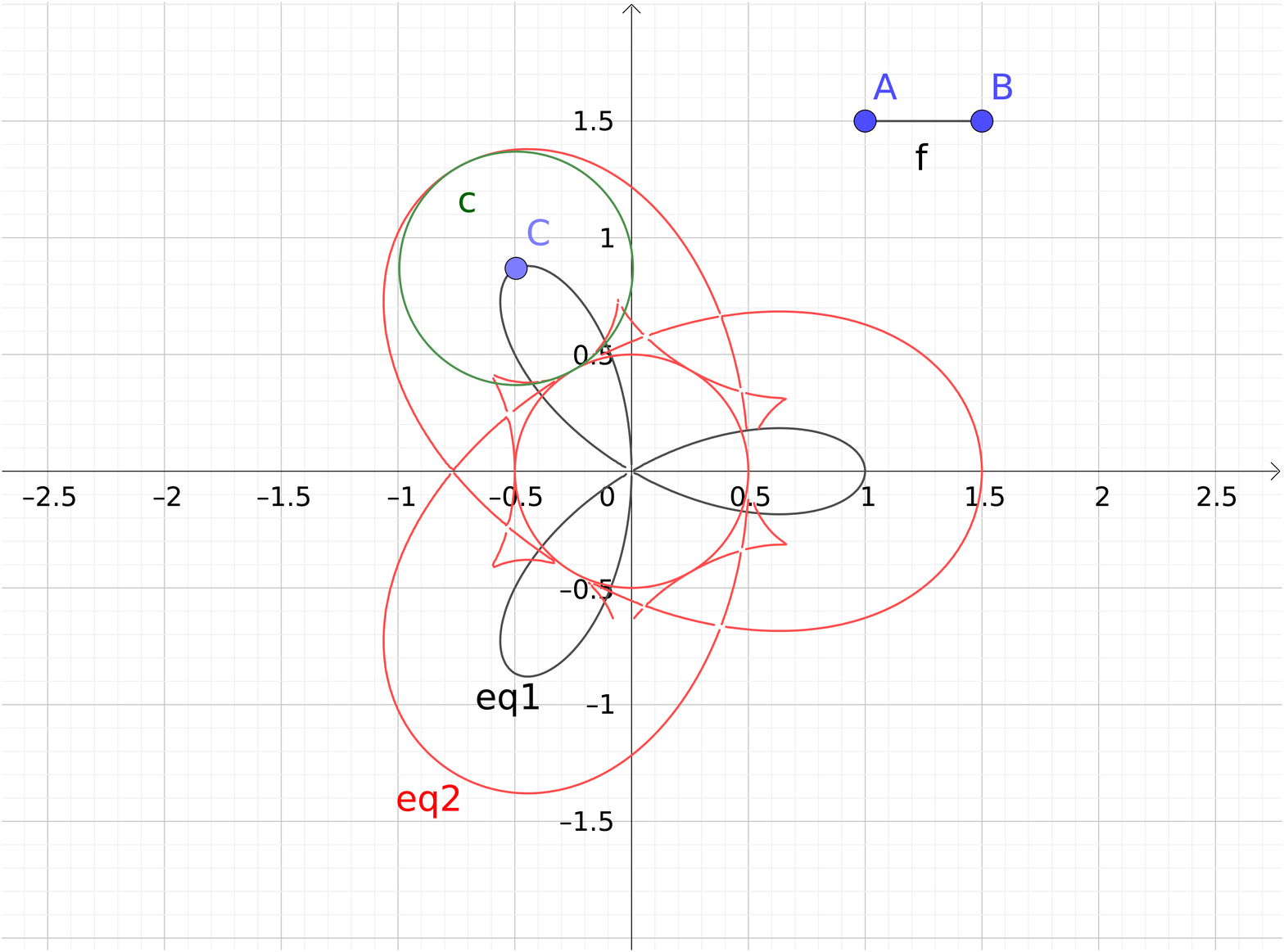,height=5cm}
\caption{The distance $0.5$ offset of the regular trifolium made with GeoGebra Discovery}
\label{trifolium-exp1}
\end{figure}

The obtained figure shows some similarities with the previously obtained figures,
most visibly the external envelope has a similar geometry. On the other hand, there is
an inmost envelope that seems to be a circle. We can verify this fact by opening
the \emph{CAS View} and issue the command
\begin{center}
\texttt{Factors(LeftSide($eq2$)-RightSide($eq2$))}
\end{center}
which informs us that the obtained
result consists of two factors: one defines a circle with center $(0,0)$ and radius $1/2$, reported
as $4 x^{2} + 4  y^{2} - 1$; the other factor
is
\begin{dmath*}
O_{1/2}(x,y)=262144 x^{14} - 524288 x^{13} + 1835008 x^{12} y^{2} - 1048576 x^{11} y^{2} + 540672 x^{11} + 5505024 x^{10} y^{4}
- 2359296 x^{10} y^{2} + 49152 x^{10} + 2621440 x^{9} y^{4} + 540672 x^{9} y^{2} - 614400 x^{9} + 9175040 x^{8} y^{6}
- 5505024 x^{8} y^{4} - 5947392 x^{8} y^{2} - 352000 x^{8} + 10485760 x^{7} y^{6} - 3244032 x^{7} y^{4} + 3981312 x^{7} y^{2}
+ 570368 x^{7} + 9175040 x^{6} y^{8} - 2621440 x^{6} y^{6} - 7766016 x^{6} y^{4} + 9098240 x^{6} y^{2} + 213504 x^{6}
+ 13107200 x^{5} y^{8} - 7569408 x^{5} y^{6} - 10911744 x^{5} y^{4} - 570368 x^{5} y^{2} - 359424 x^{5} + 5505024 x^{4} y^{10}
+ 1572864 x^{4} y^{8} + 1867776 x^{4} y^{6} + 1390080 x^{4} y^{4} - 2059776 x^{4} y^{2} - 277344 x^{4} + 7340032 x^{3} y^{10}
- 5947392 x^{3} y^{8} + 13320192 x^{3} y^{6} - 2851840 x^{3} y^{4} + 718848 x^{3} y^{2} + 62208 x^{3} + 1835008 x^{2} y^{12}
+ 786432 x^{2} y^{10} + 2998272 x^{2} y^{8} - 7244800 x^{2} y^{6} + 2440704 x^{2} y^{4} - 554688 x^{2} y^{2} + 139968 x^{2}
+ 1572864 x y^{12} - 1622016 x y^{10} + 516096 x y^{8} - 1711104 x y^{6} + 1078272 x y^{4} - 186624 x y^{2}
+ 262144 y^{14} - 262144 y^{12} - 638976 y^{10} + 815360 y^{8} - 86528 y^{6} - 277344 y^{4} + 139968 y^{2} - 19683,
\end{dmath*}
defining again a curve of degree 14.

We emphasize here that, in the background, the analytic method was used by GeoGebra, that is, computing a Jacobi
determinant and then using elimination to obtain an equation in two variables, defining a plane curve.
(See \cite{k17} and \cite{k18} for more details.)
In other words: GeoGebra performed an \textit{automated proof} on the validity of the formula that describes
$O_{1/2}(x,y)$. What is more, GeoGebra's symbolic capability ensured that the conjectured formula was
also obtained in an automated way. (In fact, these two steps: conjecture and proof were
performed in one operation.)
The computation took less than half of a second on a modern computer.

When choosing $f=1$ or other values, similar results can be obtained. For example, when
changing $f$ to $1$ the result $(x^2+y^2-1)\cdot F_2(x,y)$ will be delivered, that is,
again a product of a circle and the offset. In this case is the computation, however, somewhat slower.

As a consequence, dynamic investigation of the set of offsets for various distances
is a bit inconvenient in GeoGebra.
When fixing point $A$ and dragging point $B$ we reach computation speed
between and 1.34 and 1.35 frame per second (FPS)\footnote{Testing was performed
on Ubuntu Linux 18.04, Intel(R) Core(TM) i7-4770 CPU @ 3.40GHz. See
\url{https://prover-test.geogebra.org/job/GeoGebra_Discovery-art-plottertest/72/artifact/fork/geogebra/test/scripts/benchmark/art-plotter/html/all.html}
for a detailed output of the benchmarking suite, test cases \texttt{trifolium-offset1.ggb}, \ldots,
\texttt{trifolium-offset4.ggb}. These results are valid for
the native version of GeoGebra, that is, GeoGebra Classic 5. The web version, that is, GeoGebra Classic 6,
underperforms this speed between 0.23 and 0.24 FPS.}, even if the same distance is defined
but with different input points---this difference comes from the very different computational
character of the slightly different algebraic translations.
In fact, for an enjoyable animation the user could expect at least 5 FPS,
so further speedup of the computation could be addressed in future work.

On the other hand, slow movement of point $B$ can show how the set of offsets look like.
Fig.~\ref{trifolium-exp3} presents how the curves change while $B$ is dragged from $(1.5,1.5)$ to $(1,0.5)$.
(Coloring of point $B$ and the curve $eq2$ depends on the length of $f$: we used the RGB-components $(f,1-f,1)$.)
An online version of this experiment can be found at \cite{offsets-of-a-trifolium}.

\begin{figure}[htbp]
\centering
\epsfig{file=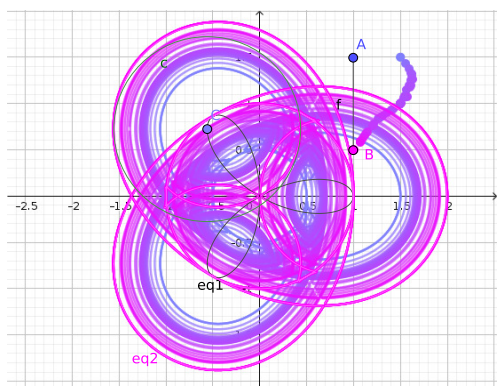,height=5cm}
\caption{Various offsets of the trifolium for distances $f$ between $0.5$ and $1$}
\label{trifolium-exp3}
\end{figure}

\section{Conclusion}
We derived the offset curve of a regular trifolium by using different approaches.
Automated reasoning based on \textit{elimination} played an important role in our results.
In our last experiment we highlighted that a point-and-click approach by using
a Dynamic Geometry System can merge conjecturing and proving to a single step,
and dragging some input points a large set of theorems can be proven in a novel way.

\section{Acknowledgments}
First author was partially supported by the CEMJ Chair at JCT 2019-2020.
Second author was partially supported by a grant MTM2017-88796-P from the
Spanish MINECO (Ministerio de Economia y Competitividad) and the ERDF
(European Regional Development Fund).
Special thanks to Bernard Parisse
for his efforts on improving the Giac Computer Algebra System
especially for this research.

\bibliographystyle{elsarticle-harv}

\end{document}